\documentclass[12pt]{article}
\usepackage{amscd, enumerate, amsmath, amssymb}
\usepackage{tikz-cd}

\usepackage{amsthm}
\usepackage{url}

\newtheorem{thm}{Theorem}[section]

\newtheorem{lem}[thm]{Lemma}
\newtheorem{rem}[thm]{Remark}
\newtheorem{ex}[thm]{Example}
\newtheorem{cor}[thm]{Corollary}
\newtheorem{prop}[thm]{Proposition}

\newtheorem{acknowledgment}{Acknowledgment}

\numberwithin{equation}{section}

\newcommand{\M}{{\mathcal M}}
\newcommand{\cL}{{\mathcal L}}
\newcommand{\cP}{{\mathcal P}}

\newcommand{\C}{{\mathbb C}}
\newcommand{\R}{{\mathbb R}}
\newcommand{\D}{{\mathbb D}}
\newcommand{\bT}{{\mathbb T}}

\newcommand{\bH}{{\mathbb H}}

\newcommand{\bI}{{\mathbb I}}
\newcommand{\bJ}{{\mathbb J}}

\newcommand{\cH}{{\mathcal H}}
\newcommand{\cK}{{\mathcal K}}

\newcommand{\ran}{\operatorname{ran}}

\newcommand{\Real}{\operatorname{Re}}

\newcommand{\la}{\langle}
\newcommand{\ra}{\rangle}
\newcommand{\p}{\partial}
\newcommand{\lam}{\lambda}

\begin{document}
\title{L\"{o}wner equations and de Branges--Rovnyak spaces}
\author{Michio Seto\\
Department of Mathematics, \\
National Defense Academy of Japan\\
{\small 
{\it E-mail address}: {\tt mseto@nda.ac.jp}}\\
} 
\maketitle
\noindent

\begin{abstract}  
We study one-parameter families of de Branges--Rovnyak spaces associated with L\"{o}wner equations. 
The purpose of this paper is to give a new approach to a problem discussed by Sarason, Fricain--Hartmann--Ross, Fricain--Mashreghi, and
{\L}anucha--Nowak.
Our approach is based on a little-known technique, ``\textit{L\"{o}wner expansions}'', 
developed by the de Branges school. 
We recast ``\textit{L\"{o}wner expansions}'' in the theory of direct integrals and integral operators. 
Further, our framework also provides a natural setting for  
applying the theory of the chordal L\"{o}wner equation to the theory of de Branges--Rovnyak spaces.  
As an application, we prove Fourier-type expansion theorems induced by the L\"{o}wner equations. 
\end{abstract}

\begin{center}
2020 Mathematical Subject Classification: Primary 47B32, 30C99.\\
keywords: de Branges--Rovnyak space, L\"{o}wner equation, reproducing kernel Hilbert space. 
\end{center}

\section{Introduction}
One of the highlights of de Branges' original proof of the Bieberbach conjecture is 
its ingenious combination of the L\"{o}wner theory 
and Hilbert space theory (see de Branges~\cite{de Branges2, de Branges3} and Vasyunin--Nikol'ski\u{\i}~\cite{VN1, VN2}). 
Although the functional analysis view of de Branges' proof is largely absent from  
the standard expositions (for example, Conway~\cite{Conway} and Rosenblum--Rovnyak~\cite{RR}),  
Ghosechowdhury continued to search for 
interaction between 
the L\"{o}wner theory and Hilbert space theory in \cite{G1, G2}, 
and obtained a Fourier-type expansion theorem called ``\textit{L\"{o}wner expansions}'' (see also the title of \cite{de Branges2}). 
More precisely, he showed continuous resolution 
of de Branges--Rovnyak spaces along the L\"{o}wner equation. 
However, perhaps because it was written in the language and framework of the de Branges school, 
its results do not appear to be widely known.
One of the original motivations for our project was to recast ``\textit{L\"{o}wner expansions}''. 
For this purpose, following Ando~\cite{Ando}, we adopt direct integrals and integral operators. 
As a consequence, we give not only a proof of Ghosechowdhury's theorem with our framework (Theorem \ref{thm:5-2}), 
but also several new results. 
For example, let $B_t(z)=B(t,a,z)$ $(0\leq a \leq t\leq b)$ be the L\"{o}wner semigroup 
associated with the Koebe function, and 
let $\cH(B_b)$ denote the de Branges--Rovnyak space induced by $B_b$.  
Then, 
\[
\log \frac{1-B_b(z)}{1-z}
\]
belongs to $\cH(B_b)$ (Example \ref{ex:3-7}). 
Moreover, we prove the following representation of $\cH(B_b)$ (Theorem \ref{thm:4-6} and Corollary \ref{cor:4-7}): 
\[
\cH(B_b)=\bigvee \left\{\exp\left(\int_a^b \dfrac{B_t(z)}{1+B_t(z)}f(t)\ dt\right): f\in L^2([a,b])\right\},
\]
where $\vee$ denotes the closed vector span. 
Thus, we give a new approach to a problem discussed in (II-8) of Sarason~\cite{Sarason}, 
Fricain--Hartmann--Ross~\cite{FHR}, Sections 18.6 and 23.8 of Fricain--Mashreghi~\cite{FM}, and {\L}anucha--Nowak~\cite{LN}.
Further, our framework also provides a natural setting for  
applying the theory of the chordal L\"{o}wner equation (see 
Hasebe--Hotta--Murayama~\cite{HHM} and Schlei{\ss}inger~\cite{Sch}) to the theory of de Branges--Rovnyak spaces. 
In fact, as an application of the second main theorem (Theorem \ref{thm:6-3}), 
we obtain the following representation of a one-parameter family of de Branges--Rovnyak spaces
(Corollary \ref{cor:6-4}):
\[
\cH(\psi_b)=\bigvee \left\{\dfrac{1-\psi_b(z)}{1-z}\exp\left(\int_a^bi\dfrac{\psi_t(z)-1}{\psi_t(z)+1}f(t)\ dt\right): f\in L^2([a,b])  \right\},
\]
where we set 
\[
\psi_t(z)=\dfrac{\sqrt{\left(i\dfrac{1+z}{1-z}\right)^2-2(t-a)}-i}
{\sqrt{\left(i\dfrac{1+z}{1-z}\right)^2-2(t-a)}+i}\quad (0\leq a \leq t\leq b).
\]

This paper is organized as follows. 
Sections 2 and 3 are the preliminary parts, where direct integrals of reproducing kernel Hilbert spaces and 
integral operators defined on those spaces are introduced, respectively. 
In Section 4, 
we investigate the structure of de Branges--Rovnyak spaces associated with the radial L\"{o}wner equation. 
In particular, we give a proof of Ghosechowdhury's theorem with our framework. 
In Section 5, we introduce auxiliary spaces associated with Pick functions for Section 6. 
In Section 6, 
we investigate the structure of de Branges--Rovnyak spaces associated with the chordal L\"{o}wner equation. 

\section{Direct integrals}\label{sec:1}

In this section, we arrange the theory of direct integrals to fit the L\"{o}wner theory. 
\subsection{Setting}
Let $\Omega$ be a domain in the complex plane $\C$, and let $(X,\sigma)$ be a $\sigma$-finite measure space.  
We consider a family $\{k(x,z,w)\}_{x\in X}$ consisting of kernel functions (positive semidefinite functions) on $\Omega\times \Omega$. 
Let $C$ denote a compact subset of $\Omega$. 
Then, we will write 
\[
M_C(x)=\sup_{z\in C}k(x,z,z).
\] 
Throughout this section, we assume the following four conditions: 
\begin{enumerate}
\item[(A)] $k(x,z,w)$ is measurable with respect to $x$ when $z$ and $w$ are fixed, 
\item[(B)] $k(x,z,w)$ is holomorphic on $\Omega$ with respect to $z$ when $x$ and $w$ are fixed, 
\item[(C)] $M_C(x)$ is integrable with respect to $\sigma$ for any $C$. 
In particular, $k(x,z,z)$ is integrable for any fixed $z$. 
\item[(D)] $M_C(x)$ is finite for any $C$ when $x$ is fixed. 
\end{enumerate}
In what follows, 
to avoid confusion, complex variables will be denoted such as $z$ and $w$, 
and specific points in $\Omega$ will be denoted such as $\lam$ and $\mu$. 
We will write $k_{\lam}(x,z)=k(x,z,\lam)$ for any $\lam$ in $\Omega$. 

\subsection{$\cH(x)$}
Let $\cH(x)$ denote the reproducing kernel Hilbert space generated 
by the kernel function $k(x,z,w)$ for a fixed point $x$ in $X$. 

\begin{prop}\label{prop:4-1} 
Any function in $\cH(x)$ is holomorphic on $\Omega$. 
\end{prop}

\begin{proof}
For any function $f(x,z)$ in $\cH(x)$, 
there exists a sequence $\{f_n(x,z)\}_n$ consisting of linear combinations of reproducing kernels of $\cH(x)$ 
such that 
\[
\lim_{n\to \infty}\|f_n(x,\cdot)-f(x,\cdot)\|_{\cH(x)}=0.
\]
Then, for any $\lam$ in $\Omega$, we have that 
\begin{align*}
|f_n(x,\lam)-f(x,\lam)|
&=|\la f_n(x,\cdot)-f(x,\cdot),k_{\lam}(x,\cdot) \ra_{\cH(x)}|\\
&\leq \|f_n(x,\cdot)-f(x,\cdot)\|_{\cH(x)}\|k_{\lam}(x,\cdot)\|_{\cH(x)}\\
&\to 0\quad(n\to \infty).
\end{align*}
It follows from Assumption (D) that  
this convergence is uniform on any compact subset of $\Omega$.   
Hence, by Assumption (B), $f(x,z)$ is holomorphic on $\Omega$ with respect to variable $z$. 
\end{proof}

\subsection{Riesz--Fischer type lemma}
Let $B(X)$ denote the set of all bounded measurable functions defined everywhere on $X$. 
We need the following auxiliary space consisting of functions with two variables $x$ and $z$: 
\[
\cK=\left\{f:f(x,z)=\sum_{j=1}^nf_j(x)k_{\lam_j}(x,z),\ f_j\in B(X),\ \lam_j\in \Omega,\ n\geq 1\right\}.
\]
Then, $\cK$ is a vector space and invariant under multiplication by functions in $B(X)$. 
We shall mention the following rather trivial fact as a proposition. 
\begin{prop}
Let $f$ be any function in $\cK$. 
Then, $\|f(x,\cdot)\|_{\cH(x)}$ is a square integrable function with variable $x$. 
\end{prop}

\begin{proof}
We set $f(x,z)=\sum_{j=1}^nf_j(x)k_{\lam_j}(x,z)$. 
Then, we have that 
\begin{align*}
\|f(x,\cdot)\|_{\cH(x)}^2&=\sum_{i,j=1}^n\overline{f_i(x)}f_j(x)\la k_{\lam_j}(x,\cdot),k_{\lam_i}(x,\cdot) \ra_{\cH(x)}\\
&=\sum_{i,j=1}^n\overline{f_i(x)}f_j(x)k(x,\lam_i,\lam_j). 
\end{align*}
Hence, by Assumption (A), $\|f(x,\cdot)\|_{\cH(x)}^2$ is measurable.  
Moreover, by Assumption (C), $\|k_{\lam_j}(x,\cdot)\|_{\cH(x)}$ is square integrable. 
Hence, we have that 
\begin{align*}
\int_X \|f(x,\cdot)\|_{\cH(x)}^2\ d\sigma(x)
&\leq \int_X\sum_{i,j=1}^n\left|\overline{f_i(x)}f_j(x)k(x,\lam_i,\lam_j)\right|\ d\sigma(x)\\
&\leq \sum_{i,j=1}^n\|f_i\|_{\infty}\|f_j\|_{\infty}\int_X\|k_{\lam_i}\|_{\cH(x)}\|k_{\lam_j}\|_{\cH(x)}\ d\sigma(x)\\
&<\infty.
\end{align*}
This concludes the proof.
\end{proof}

We need the following Riesz--Fischer type lemma. 

\begin{lem}\label{lem:4-1}
Let $\{f_n\}_n$ be a sequence in $\cK$. 
If 
\[
\int_X\|f_n(x,\cdot)-f_m(x,\cdot)\|_{\cH(x)}^2\ d\sigma(x)\to 0\quad(n,m\to \infty),
\] 
then, 
for almost all $x$,   
there exists a function $f(x, z)$ in $\cH(x)$ such that 
\[
\int_X\|f(x,\cdot)\|_{\cH(x)}^2\ d\sigma(x)<\infty
\]
and
\[
\int_X\|f_n(x,\cdot)-f(x,\cdot)\|_{\cH(x)}^2\ d\sigma(x)\to 0\quad(n\to \infty).
\] 
\end{lem}

\begin{proof}
Since this lemma is essentially the same as the Riesz--Fischer theorem, 
we give a sketch of the proof. 
It follows from the assumption that 
there exists a subsequence $\{f_{n_k}\}_k$ of $\{f_n\}_n$ such that 
\[
\int_X \|f_{n_{k+1}}(x,\cdot)-f_{n_k}(x,\cdot)\|_{\cH(x)}^2\ d\sigma(x)<\frac{1}{2^{2k}}. 
\]
Then, since 
\[
g(x)=\|f_{n_1}(x,\cdot)\|_{\cH(x)}+\sum_{k=1}^{\infty}\|f_{n_{k+1}}(x,\cdot)-f_{n_k}(x,\cdot)\|_{\cH(x)} 
\]
is square integrable by the monotone convergence theorem. 
In particular, $g(x)$ is finite for almost all $x$ in $X$. 
Hence,
we have that
\[
\|f_{n_{\ell}}(x,\cdot)-f_{n_k}(x,\cdot)\|_{\cH(x)}
\leq \sum_{j=k}^{\ell-1}\|f_{n_{j+1}}(x,\cdot)-f_{n_j}(x,\cdot)\|_{\cH(x)}\to 0\quad(k,\ell\to\infty)
\]
for almost all $x$ in $X$. 
Therefore, 
for almost all $x$,   
there exists a function $f(x, z)$ in $\cH(x)$ 
such that 
\[
\lim_{k\to \infty}\|f_{n_k}(x,\cdot)-f(x,\cdot)\|_{\cH(x)}=0. 
\]
Then, 
since $\|f_{n_k}(x,\cdot)\|_{\cH(x)}\leq g(x)$, 
by Lebesgue's dominated convergence theorem, we have 
\[
\int_X\|f(x,\cdot)\|_{\cH(x)}^2\ d\sigma(x)<\infty.
\]
Further, 
by Fatou's lemma, we have 
\[
\int_X\|f_n(x,\cdot)-f(x,\cdot)\|_{\cH(x)}^2\ d\sigma(x)\to 0\quad(n\to \infty).
\] 
\end{proof}

\subsection{$L^2(\cH(x))$}
Let $L^2(\cH(x))$ denote the set of all functions obtained in the manner of Lemma \ref{lem:4-1}. 
Then, for any functions $f$ and $g$ in $L^2(\cH(x))$,  
setting 
\[
\|f\|_{L^2(\cH(x))}=\left(
\int_X \|f(x,\cdot)\|_{\cH(x)}^2\ d\sigma(x)
\right)^{1/2}
\]
and
\[
\la f,g \ra_{L^2(\cH(x))}=\int_X\la f(x,\cdot),g(x,\cdot) \ra_{\cH(x)}\ d\sigma(x), 
\]
the fundamental inequalities
\[
|\la f, g\ra_{L^2(\cH(x))}|\leq \|f\|_{L^2(\cH(x))}\|g\|_{L^2(\cH(x))}
\]
and
\[
\|f+g\|_{L^2(\cH(x))}\leq \|f\|_{L^2(\cH(x))}+ \|g\|_{L^2(\cH(x))}
\]
are valid. 
In what follows, two functions 
$f$ and $g$ in $L^2(\cH(x))$ will be identified if 
$\|f-g\|_{L^2(\cH(x))}=0$. 
We summarize the basic properties of $L^2(\cH(x))$ as follows: 

\begin{thm}\label{thm:2-4}
$L^2(\cH(x))$ is a $B(X)$-Hilbert module consisting of functions with two variables $x$ and $z$. 
Moreover, let $f$ be any function in $L^2(\cH(x))$. Then, 
$f(x,z)$ is a measurable function with variable $x$ when $z$ is fixed and 
$f(x,z)$ is a holomorphic function on $\Omega$ with variable $z$ for almost all $x$. 
\end{thm}

\begin{proof}
The proof of the first half is the same as that of Lemma \ref{lem:4-1}. 
We give the proof of the second half. 
Let $\{f_n\}_n$ be a sequence in $\cK$ such that 
\[
\lim_{n\to \infty}\|f_n-f\|_{L^2(\cH(x))}=0. 
\]   
Moreover, without loss of generality, we may assume that 
\[
\lim_{n\to \infty}\|f_n(x,\cdot)-f(x,\cdot)\|_{\cH(x)}=0
\]
for almost all $x$ in $X$.
Then, for any $\lam$ in $\Omega$, we have that
\[
|f_n(x,\lam)-f(x,\lam)|
\leq \|f_n(x,\cdot)-f(x,\cdot)\|_{\cH(x)}\|k_{\lam}(x,\cdot)\|_{\cH(x)}
\to 0\quad(n\to \infty).
\]
By Assumptions (A), (B) and (D), we have the conclusion. 
\end{proof}

\begin{ex}\label{ex:2-5}\rm 
Suppose that $h(x,z)$ is a nonvanishing function on $X\times \Omega$ and 
$h(x,z)$ belongs to $B(X)$ if $z$ is fixed. 
We shall consider kernel function $k(x,z,w)=\overline{h(x,w)}h(x,z)$. Then, 
it is easy to see that 
\[
\cK=\{f:f(x,z)=g(x)h(x,z),\ g\in B(X)\}
\] 
and 
\[
\|f\|_{L^2(\cH(x))}^2=\int_X|g(x)|^2\|h(x,\cdot)\|_{\cH(x)}^2\ d\sigma(x)
\]
for any $f(x,z)=g(x)h(x,z)$ in $\cK$. 
Moreover, since
\[
|h(x,w)|^2=k(x,w,w)=\left\|\overline{h(x,w)}h(x,\cdot)\right\|_{\cH(x)}^2=|h(x,w)|^2\|h(x,\cdot)\|_{\cH(x)}^2, 
\]
we have that $\|h(x,\cdot)\|_{\cH(x)}^2=1$. 
Hence, if $\{f_n\}_n$ is a Cauchy sequence in $L^2(\cH(x))$, then 
so is $\{g_n\}_n$ in $L^2(X)$, where we set $f_n(x,z)=g_n(x)h(x,z)$.  
This observation concludes that 
\[
L^2(\cH(x))=\{f:f(x,z)=g(x)h(x,z),\ g \in L^2(X)\}.
\]
\end{ex}

\section{Integral operators}

Let $\{B(x,z)\}_{x\in X}$ be a family consisting of holomorphic functions on $\Omega$ satisfying
\begin{enumerate}
\item[(E)] $\displaystyle \|B\|_{\infty,C}=\sup_{x\in X,z \in C}|B(x,z)|$ is finite for any compact subset $C$ of $\Omega$. 
\end{enumerate}
In this section, we consider the integral operator defined as follows:
\[
\bI_B: f\mapsto \int_XB(x,z)f(x,z)\ d\sigma(x)\quad (f\in L^2(\cH(x))).
\]

\subsection{Basic properties}

First, we note that $(\bI_B f)(\lam)$ is finite for any $\lam$ in $\Omega$. 
Indeed, by Assumption (E), we have that
\begin{align*}
\int_X|B(x,\lam)f(x,\lam)|\ d\sigma(x)
&\leq \|B\|_{\infty,\{\lam\}}\int_X|f(x,\lam)|\ d\sigma(x)\\
&\leq \|B\|_{\infty,\{\lam\}}\int_X\|f(x,\cdot)\|_{\cH(x)}\|k_{\lam}(x,\cdot)\|_{\cH(x)}\ d\sigma(x)\\
&\leq \|B\|_{\infty,\{\lam\}}\|f\|_{L^2(\cH(x))}\|k_{\lam}\|_{L^2(\cH(x))}.
\end{align*}
Moreover, we have that
\begin{equation}\label{eq:2-1}
|(\bI_B f)(\lam)-(\bI_B g)(\lam)|\leq \|B\|_{\infty,\{\lam\}}\|f-g\|_{L^2(\cH(x))}\|k_{\lam}\|_{L^2(\cH(x))}
\end{equation}
for any $f$ and $g$ in $L^2(\cH(x))$. 

\begin{lem}\label{lem:1-2-0}
For any $\lam$ in $\Omega$, $\bI_B k_{\lam}$ is a holomorphic function on $\Omega$. 
\end{lem}

\begin{proof}
Let $a$ be a point in $\Omega$ and 
let $c_j(x,\lam)$ denote the $j$-th coefficient of the power series expansion centered at $a$ of $B(x,z)k_{\lam}(x,z)$. 
Then, 
choosing an appropriate $R>0$ and setting $C=\{a+Re^{i\theta}:0\leq \theta <2\pi\}$, 
we have that 
\begin{align*}
|c_j(x,\lam)|
&\leq \frac{1}{2\pi}\int_0^{2\pi}\frac{|B(x,a+Re^{i\theta})k_{\lam}(x,a+Re^{i\theta})|}{R^j}\ d\theta\\
&\leq \frac{\|B\|_{\infty,C}}{2\pi}\int_0^{2\pi}\frac{\|k_{a+Re^{i\theta}}(x,\cdot)\|_{\cH(x)}\|k_{\lam}(x,\cdot)\|_{\cH(x)}}{R^j}\ d\theta\\
&\leq \frac{\|B\|_{\infty,C}\sqrt{M_C(x)}\|k_{\lam}(x,\cdot)\|_{\cH(x)}}{R^j}.
\end{align*}
Hence, if $|z-a|<r<R$, then
\begin{align*}
&\int_X \left(\sum_{j=0}^{\infty}|c_j(x,\lam)(z-a)^j|\right)\ d\sigma(x)\\
&\leq \int_X \left(\sum_{j=0}^{\infty}\frac{\|B\|_{\infty,C}\sqrt{M_C(x)}\|k_{\lam}(x,\cdot)\|_{\cH(x)}}{R^j}r^j\right)\ d\sigma(x)\\
&=\frac{R}{R-r}\|B\|_{\infty,C}\left(\int_X M_C(x)\ d\sigma(x)\right)^{1/2}\|k_{\lam}\|_{L^2(\cH(x))}\\
&<\infty
\end{align*}
by Assumption (C). 
Hence, by Fubini's theorem, we have that
\begin{align*}
(\bI_B k_{\lam})(z)
&=\int_X B(x,z)k_{\lam}(x,z)\ d\sigma(x)\\
&=\int_X \left(\sum_{j=0}^{\infty}c_j(x,\lam)(z-a)^j\right)\ d\sigma(x)\\
&=\sum_{j=0}^{\infty}\left(\int_X c_j(x,\lam)\ d\sigma(x) \right)(z-a)^j.
\end{align*}
This concludes the proof. 
\end{proof}

\begin{lem}\label{lem:2-2}
The range space of $\bI_B$ is a subspace of the vector space consisting of holomorphic functions on $\Omega$ 
and $\ker \bI_B$ is closed. 
\end{lem}

\begin{proof}
Let $f$ be any function in $L^2(\cH(x))$. Then, by definition, there exists a sequence $\{f_n\}_n$ in $\cK$ such that 
\[
\lim_{n\to \infty}\|f_n-f\|_{L^2(\cH(x))}=0.
\]
Then, by (\ref{eq:2-1}), $(\bI_B f_n)(\lam)$ converges to $(\bI_B f)(\lam)$ for every $\lam$ in $\Omega$. 
Moreover, by Assumptions (D) and (E), this convergence is uniform on any compact subset of $\Omega$. 
Hence,  by Lemma \ref{lem:1-2-0}, $(\bI_B f)(z)$ is a holomorphic function on $\Omega$. 
Next, 
let $f$ be a function in the closure of $\ker \bI_B$ and let $\{f_n\}_n$ be a sequence in $\ker \bI_B$ such that 
\[
\lim_{n\to \infty}\|f_n-f\|_{L^2(\cH(x))}=0.
\]
Then, 
by (\ref{eq:2-1}), we have that
\begin{align*}
|(\bI_B f)(\lam)|
&=|(\bI_B f)(\lam)-(\bI_B f_n)(\lam)|\\
&\leq \|B\|_{\infty,\{\lam\}} \|f-f_n\|_{L^2(\cH(x))}\|k_{\lam}\|_{L^2(\cH(x))}\\
&\to 0\quad(n\to \infty)
\end{align*}
for any $\lam$ in $\Omega$. 
Hence $\ker \bI_B$ is closed. 
\end{proof}

By Lemma \ref{lem:2-2} and the pull-back construction, 
we can endow the range space of $\bI_B$ with a Hilbert space structure as follows: 
\begin{align*}
\M(\bI_B; L^2(\cH(x)))
&=\M(\bI_B)=\left(\ran \bI_B,\ \|\cdot\|_{\M(\bI_B)}\right),\\
\|\bI_B f\|_{\M(\bI_B)}
&=\|P_{(\ker \bI_B)^{\perp}}f\|_{L^2(\cH(x))},
\end{align*}
where $(\ker \bI_B)^{\perp}$ denotes the orthogonal complement of $\ker \bI_B$ in $L^2(\cH(x))$. 

\begin{thm}\label{thm:2-2}
$\M(\bI_B)$ is a reproducing kernel Hilbert space and 
its reproducing kernel is 
\[
\left(\bI_B \overline{B(\cdot,\lam)}k_{\lam}\right)(z)=\int_X\overline{B(x,\lam)}B(x,z)k_{\lam}(x,z)\ d\sigma(x).
\]
In particular, for any function $F$ in $\M(\bI_B)$, there exists some function $f$ in $L^2(\cH(x))$ such that 
\[
F(z)=\int_XB(x,z)f(x,z)\ d\sigma(x)
\]
and
\[
\|F\|_{\M(\bI_B)}=\min_{F=\bI_B f} \|f\|_{L^2(\cH(x))}.
\]
\end{thm}

\begin{proof}
We set $h_{\lam}(x,z)=\overline{B(x,\lam)}k_{\lam}(x,z)$. 
If $g$ belongs to $\ker \bI_B$, then
\begin{align*}
\la g, h_{\lam} \ra_{L^2(\cH(x))}
&=\int_{X}\la g(x,\cdot), h_{\lam}(x,\cdot) \ra_{\cH(x)}\ d\sigma(x)\\
&=\int_{X}B(x,\lam)g(x,\lam)\ d\sigma(x)\\
&=(\bI_B g)(\lam)\\
&=0. 
\end{align*}
This concludes that $h_{\lam}$ belongs to $(\ker \bI_B)^{\perp}$. 
Hence, for any function $f$ in $L^2(\cH(x))$, we have that 
\begin{align*}
\la \bI_B f, \bI_B h_{\lam} \ra_{\M(\bI_B)}
&=\la f,  h_{\lam}\ra_{L^2(\cH(x))}\\
&=\int_{X}\la f(x,\cdot),  h_{\lam}(x,\cdot)\ra_{\cH(x)}\ d\sigma(x)\\
&=\int_{X}B(x,\lam)f(x,\lam)\ d\sigma(x)\\
&=(\bI_B f)(\lam).
\end{align*}
Therefore, 
$\bI_B h_{\lam}$ is the reproducing kernel of $\M(\bI_B)$. 
\end{proof}

\begin{rem}\rm 

In de Branges~\cite{de Branges2}, the following type of integral operator was studied: 
\[
\bJ_B: f\mapsto \int_Xf(x,B(x,z))\ d\sigma(x)\quad (f\in L^2(\cH(x))). 
\]
In our framework, 
under suitable conditions on $k$ and $B$, 
it is shown that 
$(\bJ_Bf)(\lam)$ is finite for any $\lam$ in $\Omega$ and 
\[
|(\bJ_B f)(\lam)-(\bJ_B g)(\lam)|\leq \|f-g\|_{L^2(\cH(x))}
\left( \int_Xk(x,B(x,\lam),B(x,\lam))\ d\sigma(x)\right)^{1/2}
\]
for any $f$ and $g$ in $L^2(\cH(x))$. It follows from this estimate that 
$\M(\bJ_B)$ is a reproducing kernel Hilbert space with kernel function 
\[
\int_Xk(x,B(x,z),B(x,\lam))\ d\sigma(x). 
\]
\end{rem}

\subsection{Examples}

\subsubsection{Paley--Wiener spaces} 
Let $A$ be a positive constant. 
We consider kernel function 
\[
\frac{\sin (2\pi A(z-\overline{w}))}{\pi(z-\overline{w})}
\]
defined on $\C\times \C$. 
Then, the reproducing kernel Hilbert space generated by this kernel function  
is a Paley--Wiener space, which will be denoted as $PW_A$. 
In Theorem \ref{thm:2-2}, 
we consider the case where $k(t,z,w)=1$, $X=[-A,A]$, $d\sigma=dt$ and $B(t,z)=e^{-2\pi i zt}$. Then, 
$\cH(t)=\C$, 
$L^2(\cH(t))=L^2([-A,A])$. 
Hence, 
\[
\bI_B: 
f(t,z) \mapsto F(z)=\int_{-A}^A B(t,z)f(t)\ dt \quad (f\in L^2([-A,A]))
\]
is the time-limited Fourier transform. 
Moreover, 
$\M(\bI_B; L^2([-A,A]))=PW_A$ as reproducing kernel Hilbert spaces. 
To see this, by Theorem \ref{thm:2-2}, it suffices to show that $\M(\bI_B)$ has the same reproducing kernel as that of $PW_A$.   
Setting $h_{\lam}(t,z)=\overline{B(t,\lam )}k_{\lam}(t,z)$, we have that 
\[
(\bI_B h_{\lam})(z)
=\int_{-A}^AB(t,z)h_{\lam}(t,z)\ dt\\
=\int_{-A}^Ae^{-2\pi i t\left(z-\overline{\lam}\right)}\ dt\\
=\frac{\sin \left(2\pi A\left(z-\overline{\lam}\right)\right)}{\pi\left(z-\overline{\lam}\right)}. 
\]
Thus, we have $\M(\bI_B)=PW_A$.

\subsubsection{Herglotz spaces}\label{sec:1-1-6} 
Let $\varphi$ be a Herglotz function, that is, 
$\varphi$ is a holomorphic function on $\D$ and $\operatorname{Re}\varphi(z)\geq 0$ and $\varphi(0)=1$.  
Then, 
$\varphi$ has the following integral representation: 
\[
\varphi(z)=\int_{\bT}\frac{1+\xi z}{1-\xi z}\ d\mu(\xi),
\]
where $\mu$ is a probability measure on the unit circle $\bT$. 
The integrand 
\[
\varphi_{\xi}(z)=\frac{1+\xi z}{1-\xi z}
\]
is a typical example of a Herglotz function. 
Then, since
\[
\frac{\overline{\varphi_{\xi}(w)}+\varphi_{\xi}(z)}{1-\overline{w}z}=
\frac{2}{\left(1-\overline{\xi w}\right)(1-\xi z)},
\]
we have

\begin{prop}\label{prop:1-3}
Let $\varphi$ be a Herglotz function. Then, 
\[
\dfrac{\overline{\varphi(w)}+\varphi(z)}{1-\overline{w}z}
\]
is a kernel function on $\D\times \D$.
\end{prop}

The reproducing kernel Hilbert space generated by the kernel function in Proposition \ref{prop:1-3} 
is called a Herglotz space, which will be denoted as $\cL(\varphi)$. 
Herglotz spaces were introduced and studied by de Branges--Rovnyak~\cite{dBR} 
(see also Ghosechowdhury~\cite{G1,G2} and Shulman~\cite{S}). 
Next, we consider the integral operator defined as follows: 
\[
\bI: 
f(\xi,z) \mapsto F(z)=\int_{\bT} f(\xi,z)\ d\mu(\xi) \quad (f\in L^2(\cL(\varphi_{\xi}))),
\]
which is obtained by setting $B\equiv 1$. 
Then,  
$\M(\bI; L^2(\cL(\varphi_{\xi})))=\cL(\varphi)$ as reproducing kernel Hilbert spaces. 
To see this, by Theorem \ref{thm:2-2}, 
it suffices to show that $\M(\bI)$ has the same reproducing kernel as that of $\cL(\varphi)$.   
Setting $k_{\lam}(\xi,z)=k(\xi,z,\lam)$, we have that 
\[
(\bI k_{\lam})(z)=\int_{\bT}k_{\lam}(\xi,z)\ d\mu(\xi)
=\int_{\bT} \frac{\overline{\varphi_{\xi}(\lam)}+\varphi_{\xi}(z)}{1-\overline{\lam} z}\ d\mu(\xi)
=\frac{\overline{\varphi(\lam)}+\varphi(z)}{1-\overline{\lam} z}. 
\]
Thus, we have $\M(\bI)=\cL(\varphi)$.

\section{L\"{o}wner expansions}
A L\"{o}wner semigroup is a family consisting of univalent self-maps 
$B_{rs}(z)=B(s,r,z)$ $(a\leq r\leq s \leq b)$ 
of $\D$ satisfying 
the nonlinear L\"{o}wner equation
\begin{equation}\label{eq:5-1}
\frac{\p B_{rs}}{\p s}(z)=-B_{rs}(z)\varphi(s,B_{rs}(z)),
\end{equation}
where $\{\varphi(t,z)\}_{0<t<\infty}$ is a family of Herglotz functions such that 
$\varphi(t,z)$ is a measurable function of $t$ for each fixed $z$ 
(see Chapter 8 in Rosenblum--Rovnyak~\cite{RR} for further details). 
In what follows, we will write $B_t=B_{at}$ for short if no confusion occurs. 
Then, for any $t\geq 0$,
\[
k(t, z,w)=\frac{\overline{\varphi(t,B_t(w))}+\varphi(t,B_t(z))}{1-\overline{w} z}
\]
is a kernel function defined on $\D\times \D$ by Proposition \ref{prop:1-3}. 
Then, trivially, Assumptions (A) and (B) are valid. 
Moreover, Assumptions (C) and (D) follow from the following estimate:
\begin{equation*}
0\leq k(t,\lam,\lam)=\dfrac{2\Real\varphi(t,B_t(\lam))}{1-|\lam|^2}\leq 
\dfrac{2(1+|B_t(\lam)|)}{(1-|\lam|^2)(1-|B_t(\lam)|)}
\end{equation*}
for any $\lam$ in $\D$ (see (8--10) on p. 193 of \cite{RR}). 
Moreover, trivially, Assumption (E) is valid for $B(t,z)=B_t(z)$.  
The following important lemma is well known to the de Branges school. 

\begin{lem}\label{lem:2-1}
Let $B_t(z)=B_{at}(z)$ $(0\leq a\leq t \leq  b)$ be a L\"{o}wner semigroup. Then, 
\[
\dfrac{\p}{\p t} \dfrac{1-\overline{B_{t}(\lam)}B_{t}(z)}{1-\overline{\lam}z}
=\dfrac{\overline{\varphi(t,B_{t}(\lam))}+\varphi(t,B_{t}(z))}{1-\overline{\lam}z}
\overline{B_{t}(\lam)}B_{t}(z).
\]
\end{lem}

\begin{proof}
By the nonlinear L\"{o}wner equation (\ref{eq:5-1}), 
we have the conclusion. 
\end{proof}

\subsection{Ghosechowdhury's theorem}

Let $\cH(B_t)$ denote the reproducing kernel Hilbert space 
generated by kernel function 
\[
\dfrac{1-\overline{B_t(w)}B_t(z)}{1-\overline{w}z},
\]
which is called the de Branges--Rovnyak space induced by $B_t$ (see de Branges--Rovnyak~\cite{dBR, dBR2}, 
Fricain--Mashreghi~\cite{FM} and Sarason~\cite{Sarason}). 
Every function in $\cH(B_t)$ has a Fourier-type integral representation. 

\begin{thm}[Ghosechowdhury~\cite{G2}]\label{thm:5-2}
Let $B_t(z)=B_{at}(z)$ $(t\geq a\geq 0)$ be a L\"{o}wner semigroup. Then, 
\[
\cH(B_b)=\C \oplus \M(\bI_B; L^2(\cL(\varphi(t,B_t)))).
\]
In particular, for any function $F$ in $\cH(B_b)$, there exists some function $f$ in $L^2(\cL(\varphi(t,B_t)))$ such that 
\[
F(z)=F(0)+\int_a^bB_{t}(z)f(t,z)\ dt
\]
and
\[
\|F\|_{\cH(B_b)}^2=|F(0)|^2+\min_{F=\bI_B f} \|f\|_{L^2(\cL(\varphi(t,B_t)))}^2.
\]
\end{thm}

\begin{proof}
We will write $h_{\lam}(t,z)=\overline{B_t(\lam)}k_{\lam}(t,z)$. 
It follows from Theorem \ref{thm:2-2} and Lemma \ref{lem:2-1} that 
\begin{align*}
\la \bI_B h_{\lam}, \bI_B h_{\mu} \ra_{\M(\bI_B)}
&=\la  h_{\lam},  h_{\mu} \ra_{L^2(\cL(\varphi(t,B_{t})))}\\
&=\int_a^b \left\la \overline{B_t(\lam)}k_{\lam}(t,\cdot),\overline{B_t(\mu)}k_{\mu}(t,\cdot) \right\ra_{\cL(\varphi(t,B_{t}))}\ dt\\
&=\int_a^b \overline{B_t(\lam)}B_t(\mu)k(t,\mu,\lam)\ dt\\
&=\int_a^b \dfrac{\p}{\p t}\dfrac{1-\overline{B_t(\lam)}B_t(\mu)}{1-\overline{\lam}\mu}\ dt\\
&=\dfrac{1-\overline{B_b(\lam)}B_b(\mu)}{1-\overline{\lam}\mu}-1.
\end{align*}
Since $1$ belongs to $\cH(B_b)$, 
\[
\dfrac{1-\overline{B_b(\lam)}B_b(z)}{1-\overline{\lam}z}-1
\]
is the reproducing kernel of $\cH(B_b)\ominus \C$, where the orthogonal complement is taken with respect to the inner product of $\cH(B_b)$. 
Hence we have that $\M(\bI_B)=\cH(B_b)\ominus \C$. 
This concludes the proof. 
\end{proof}

Combining 
Theorem \ref{thm:5-2} with the observations in \ref{sec:1-1-6}, 
we see that every function in $\cH(B_{b})$ has a double integral representation. 
Indeed, let $F$ be any function in $\cH(B_b)$. 
Then, by Theorem \ref{thm:5-2}, 
there exists some function $f$ in $L^2(\cL(\varphi(t,B_{t})))$ such that  
\[
F(z)=F(0)+\int_a^bB_t(z)f(t,z)\ dt. 
\]
Moreover, for almost all $t$, 
it follows from the observation in \ref{sec:1-1-6} 
that there exists some function $g$ in $L^2(\cL(\varphi_{\xi}(B_{t})))$ such that
\[
f(t,z)=\int_{\bT} g(\xi, t, z)\ d\mu_t(\xi).
\]
Hence we have that 
\begin{equation*}
F(z)=F(0)+\int_a^bB_t(z)\left(\int_{\bT} g(\xi, t, z)\ d\mu_t(\xi)\right)\ dt. 
\end{equation*}
We will carry on the above observation in the case where $\cH(B_t)$ is associated with the Koebe function. 
The following lemma is needed. 

\begin{lem}[de Branges--Rovnyak~\cite{dBR}]\label{lem:4-3}
Let $B$ be a holomorphic function defined on $\D$. 
If $\|B\|_{\infty}\leq 1$ and $B\neq -1$, then 
\[
V: 
f \mapsto \dfrac{\sqrt{2}}{1+B}f
\]
gives an isometry from $\cH(B)$ onto $ \cL(\varphi_{-1}(B))$. 
\end{lem}

\begin{proof}
We shall give a proof for the sake of readers. 
In this proof, we will write
\[
k_{\lam}(z)=\frac{1-\overline{B(\lam)}B(z)}{1-\overline{\lam}z}
\quad \text{and} \quad
\ell_{\lam}(z)=\dfrac{\overline{\varphi_{-1}(B(\lam))}+\varphi_{-1}(B(z))}{1-\overline{\lam}z}
\]
for any $\lam$ in $\D$. 
Then, by direct calculation, we have that 
\[
\ell_{\lam}(z)
=
\frac{2}{\left(1+\overline{ B(\lam)}\right)\left( 1+B(z)\right)}
k_{\lam}(z).
\]
It follows from this identity that 
$Vk_{\lam}$ is an element in $\cL(\varphi_{-1}(B))$ and 
\begin{align*}
\la Vk_{\lam},Vk_{\mu} \ra_{\cL(\varphi_{-1}(B))}
&=\left\la \frac{1+\overline{B(\lam)}}{\sqrt{2}}\ell_{\lam}, 
\frac{1+\overline{B(\mu)}}{\sqrt{2}}\ell_{\mu}\right\ra_{\cL(\varphi_{-1}(B))}\\
&=\frac{\left(1+\overline{B(\lam)}\right)(1+B(\mu))}{2}\ell_{\lam}(\mu)\\
&=k_{\lam}(\mu)\\
&=\la k_{\lam},k_{\mu} \ra_{\cH(B)}. 
\end{align*}
Hence we have the conclusion. 
\end{proof}

In what follows, the following one-parameter family of holomorphic functions 
will be called the Koebe functions: 
\[
\frac{e^tz}{(1-z)^2}\quad (t\geq 0).
\]
For the detailed discussion of the Koebe functions, see Section 8.1 in \cite{RR}.  

\begin{thm}\label{thm:4-3}
Let $B_t(z)=B_{at}(z)$ $(0\leq a \leq t \leq b)$ be the L\"{o}wner semigroup associated with the Koebe functions. 
Then, 
\[
\cH(B_b)=\C\oplus 
\left\{
\int_a^b\dfrac{B_s(z)}{1+B_s(z)}f(s,z)\ ds:f(s,z)\in L^2(\cH(B_t))
\right\}.
\]
\end{thm}

\begin{proof}
First, we note that $\cL(\varphi(t,B_{t}))=\cL(\varphi_{-1}(B_t))$, because, 
for the Koebe functions, 
each $\mu_t$ is the Dirac measure centered on $\xi=-1$ for all $t\geq 0$. 
Next, the operator $V$ obtained in  Lemma \ref{lem:4-3}
can be extended to an onto isometry 
\[
{\mathbb V}: L^2(\cH(B_t))\to L^2(\cL(\varphi_{-1}(B_t))).
\]
Let $F$ be any function in $\cH(B_b)$. 
Then, 
by Theorem \ref{thm:5-2}, 
there exists a function $g$ in $L^2(\cL(\varphi_{-1}(B_t)))$ such that 
\[
F(z)=F(0)+\int_a^bB_s(z)g(s, z)\ ds.
\]
Hence, setting $f=\sqrt{2}{\mathbb V}^{-1}g$, we have that 
\[
F(z)=F(0)+\int_a^b\dfrac{B_s(z)}{1+B_s(z)}f(s, z)\ ds
\]
and $f$ belongs to $L^2(\cH(B_t))$. 
Conversely, 
let $f$ be a function in $L^2(\cH(B_t))$. 
Then, $f(s,z)/(1+B_s(z))={\mathbb V f}/\sqrt{2}$ belongs to $L^2(\cL(\varphi_{-1}(B_t)))$. Hence, 
\[
F(z)=\int_a^b\dfrac{B_s(z)}{1+B_s(z)}f(s, z)\ ds
\]
is a function in $\cH(B_b)$ by Theorem \ref{thm:5-2}.
\end{proof}

\begin{ex}\label{ex:3-7}\rm 
Let $B_t(z)=B_{at}(z)$ $(0\leq a\leq t \leq b)$ be the L\"{o}wner semigroup associated with the Koebe functions. 
Then,  
\[
\log \frac{1-B_b(z)}{1-z}
\]
belongs to $\cH(B_b)$. 
Indeed, 
if 
\[
F(z)=\int_a^b \frac{B_t(z)}{1+B_t(z)}\ dt,
\]
then, $F$ belongs to $\cH(B_b)$ by Theorem \ref{thm:4-3}. 
Moreover, by the nonlinear L\"{o}wner equation (\ref{eq:5-1}), we have that 
\begin{align*}
F(z)
&=\int_a^b \frac{B_t(z)}{1+B_t(z)}\ dt\\
&=\int_a^b\frac{-1}{1-B_t(z)}\left(-B_t(z)\frac{1-B_t(z)}{1+B_t(z)}\right)\ dt\\
&=\int_a^b \frac{-1}{1-B_t(z)}\cdot \frac{\p B_t}{\p t}(z)\ dt\\
&=\int_a^b \frac{\p}{\p t}\log\left(1-B_t(z) \right)\ dt\\
&=\log \frac{1-B_b(z)}{1-z}.
\end{align*}
Here, we would like to mention that $B_t/z$ is outer by Theorem 3.17 in Duren~\cite{Duren} 
and is extreme by Figure 8.2 in Rosenblum--Rovnyak~\cite{RR} 
(cf. Question 6.6 in Fricain--Hartmann--Ross~\cite{FHR}).  
\end{ex}

\subsection{Another representation}

Let $\M(t)$ $(0\leq a\leq t \leq b)$ denote the reproducing kernel Hilbert space generated by kernel function 
\[
\int_{\bT} \frac{2\overline{B_t(w)}B_t(z)}{\left(1-\overline{\xi B_t(w)}\right)(1-\xi B_t(z))}\ d\mu_t(\xi).
\]
Using the notation of Theorem \ref{thm:2-2}, we may write 
$\M(t)=(\bI_{\sqrt{2}B_t/(1-\xi B_t)};L^2(\bT,\mu_t(\xi)))$. 
Moreover, we consider integral operator
\[
\bI: f(t,z) \mapsto \int_a^b f(t,z)\ dt\quad (f\in L^2(\M(t))). 
\]
Then, it follows from Theorem \ref{thm:2-2} that 
\begin{equation}\label{eq:4-3}
\int_a^b \left(\int_{\bT} \frac{2\overline{B_t(w)}B_t(z)}{\left(1-\overline{\xi B_t(w)}\right)(1-\xi B_t(z))}\ d\mu_t(\xi)\right)\ dt
\end{equation}
is the kernel function which generates $\M=\M(\bI;L^2(\M(t)))$.  
Let $k$ be a kernel function. 
Then, so is $\exp k$. We will write $\exp \cH_k=\cH_{\exp k}$, 
where $\cH_k$ denotes the reproducing kernel Hilbert space generated by $k$. 
For the details of $\exp \cH_k$, see Exercise (k) on p. 320 of Nikolski~\cite{Nik}, Chapter 7 of Paulsen--Raghupathi~\cite{PR}, 
Kuwahara--S~\cite{KS} and S~\cite{Seto}.  

\begin{thm}\label{thm:4-6}
Let $B_t(z)=B_{at}(z)$ $(0\leq a \leq t \leq b)$ be a L\"{o}wner semigroup. Then, 
\begin{align*}
\cH(B_b)
&=\exp \M(\bI;L^2(\M(t)))\\
&=\bigvee\left\{\exp \int_a^b g(t,z)\ dt : g(t,z)\in L^2(\M(t))\right\}.
\end{align*}
\end{thm}

\begin{proof}
Let $k$ denote the kernel function defined by (\ref{eq:4-3}). 
Then, 
for any $\lam$ in $\D$, it follows from Lemma \ref{lem:2-1} that   
\begin{align*}
\exp k(z,\lam)
&=\exp \left(\int_a^b \left(\int_{\bT} \frac{2\overline{B_t(w)}B_t(z)}{\left(1-\overline{\xi B_t(w)}\right)(1-\xi B_t(z))}\ d\mu_t(\xi)\right)\ dt\right)\\
&=\exp \left(\int_a^b\dfrac{1-\overline{\lam}z}{1-\overline{B_t(\lam)}B_t(z)}
\cdot \dfrac{\p}{\p t}\dfrac{1-\overline{B_t(\lam)}B_t(z)}{1-\overline{\lam}z}\ dt\right)
\\
&=\exp \left(\int_a^b \dfrac{\p}{\p t}\log  \dfrac{1-\overline{B_t(\lam)}B_t(z)}{1-\overline{\lam}z}\ dt\right)\\
&=\dfrac{1-\overline{B_t(\lam)}B_t(z)}{1-\overline{\lam}z}.
\end{align*} 
Since $\exp \M(\bI;L^2(\M(t)))$ is generated by $\exp k$, 
we have the conclusion. 
\end{proof}

\begin{rem}\rm 
It should be noted that kernel function
\[
\log  \dfrac{1-\overline{B_t(\lam)}B_t(z)}{1-\overline{\lam}z}
\]
appeared on p. 30 of de Branges~\cite{de Branges3}. 
\end{rem}

\begin{cor}\label{cor:4-7}\rm 
Let $B_t(z)=B_{at}(z)$ $(0\leq a \leq t \leq b)$ be the L\"{o}wner semigroup associated with the Koebe functions. 
Then, 
\begin{align*}
\cH(B_b)
&=\exp \M(\bI_{B_t/(1+B_t)};L^2([a,b]))\\
&=\bigvee \left\{ \exp\left(\int_a^b\dfrac{B_t(z)}{1+B_t(z)}f(t)\ dt\right):f\in L^2([a,b]) \right\}.
\end{align*}

\end{cor}

\begin{proof}
It follows from the assumption that 
each $\mu_t$ is the Dirac measure centered on $\xi=-1$ for all $t\geq 0$. 
Hence, we have that 
\[
\int_a^b \left(\int_{\bT} \frac{2\overline{B_t(w)}B_t(z)}{\left(1-\overline{\xi B_t(w)}\right)(1-\xi B_t(z))}\ d\mu_t(\xi)\right)\ dt
=\int_a^b  \frac{2\overline{B_t(w)}B_t(z)}{\left(1+\overline{B_t(w)}\right)(1+B_t(z))}\ dt.
\]
Therefore, we have the conclusion by 
Theorem \ref{thm:2-2} and Theorem \ref{thm:4-6}. 
\end{proof}

\begin{ex}\label{ex:4-8}\rm 
Let $B_t(z)=B_{at}(z)$ $(0\leq a \leq t \leq b)$ be the L\"{o}wner semigroup associated with the Koebe functions. 
Then, 
\[
\dfrac{1-B_b(z)}{1-z},
\]
belongs to $\cH(B_b)$. 
Indeed, 
\[
\exp \left(\int_a^b \frac{B_t(z)}{1+B_t(z)}\ dt\right)
=\exp \left(\log \dfrac{1-B_b(z)}{1-z}\right)= \dfrac{1-B_b(z)}{1-z}
\]
by the calculation in Example \ref{ex:3-7}. 
See also Example \ref{cor:5-5}. 
\end{ex}

\section{Pick functions}\label{sec:6}

Let $\varphi$ be a Pick function, that is, 
$\varphi$ is a holomorphic self-map of the upper half-plane $\bH=\{z\in \C: \operatorname{Im}z>0\}$. 
Every Pick function has a Nevanlinna representation 
\begin{equation}\label{eq:5-0}
\varphi(z)=b+cz+\frac{1}{\pi}\int_{-\infty}^{\infty}\left( \frac{1}{t-z}-\frac{t}{1+t^2}\right)\ d\mu(t),
\end{equation}
where $b$ is real, $c\geq 0$ and $\mu$ is a nonnegative Borel measure on $(-\infty,\infty)$ satisfying
\[
\int_{-\infty}^{\infty}\frac{1}{1+t^2}\ d\mu(t)<\infty,
\]
and the converse is also true (see Theorem 5.3 in Rosenblum--Rovnyak~\cite{RR}). 

\begin{ex}\label{ex:5-1}\rm
Setting 
\[
\varphi_{\xi}(z)=\frac{1}{\xi -z}\quad (\xi \in \R),
\]
this $\varphi_{\xi}$ is a typical example of a Pick function, 
which is obtained from (\ref{eq:5-0}) if we choose $\mu$ as the Dirac measure centered on $\xi$. 
\end{ex}

Let $\varphi$ be a Pick function. 
Then $\psi$ will denote the function defined by the following diagram:
\begin{equation}\label{eq:6-2}
\begin{CD}
\bH @>{\varphi}>> \bH \\
@A{T}AA    @VV{T^{-1}}V \\
\D  @>>{\psi}>  \D,
\end{CD}
\end{equation}
where 
\[
T(z)=i\dfrac{1+z}{1-z}\quad \text{and}\quad T^{-1}(w)=\frac{w-i}{w+i}.
\]
Then, $\psi$ is a holomorphic self-map of $\D$. 
Let $\cH(\psi)$ denote the de Branges--Rovnyak space induced by $\psi$, that is, 
$\cH(\psi)$ is the reproducing kernel Hilbert space generated by kernel function
\[
\frac{1-\overline{\psi(w)}\psi(z)}{1-\overline{w}z}.
\] 
Let $\alpha$ and $\beta$ be arbitrary two points in $\bH$. 
We set $\lam=T^{-1}(\alpha)$ and $\mu=T^{-1}(\beta)$. 
Then, by direct calculation, we have that 
\begin{equation}\label{eq:6-3}
\frac{\varphi(\beta)-\overline{\varphi(\alpha)}}{\beta-\overline{\alpha}}
=\frac{1-\overline{\lam}}{1-\overline{\psi(\lam)}} \cdot \frac{1-\mu}{1-\psi(\mu)} \cdot \frac{1-\overline{\psi(\lam)}\psi(\mu)}{1-\overline{\lam}\mu}.
\end{equation}
By (\ref{eq:6-3}) and Schur's theorem,
we have the following well-known fact. 
\begin{prop}\label{prop:5-2}
Let $\varphi$ be a Pick function. Then, 
\[
\frac{\varphi (z)-\overline{\varphi (w)}}{z-\overline{w}}
\]
is a kernel function on $\bH \times \bH$.
\end{prop}

Let $\cP(\varphi)$ denote
the reproducing kernel Hilbert space generated by the kernel function in Proposition \ref{prop:5-2}, 
which is essentially the same space as that introduced in Section 5 of Chapter 1 of de Branges~\cite{de Branges}. 
The following two facts can be found in Yamada~\cite{Yamada}. 
\begin{prop}\label{thm:5-4}
Let $\varphi$ be a Pick function
with Nevanlinna representation (\ref{eq:5-0}), 
and let $\M$ denote the reproducing kernel Hilbert space over $\bH$ generated by kernel function 
\[
k(z,w)=\frac{1}{\pi}\int_{-\infty}^{\infty}\frac{1}{(t-\overline{w})(t-z)}\ d\mu(t).
\]
Then, 
\[
\cP(\varphi)=
\begin{cases}
\C + \M & (c \neq 0)\\
\M & (c=0 )
\end{cases}
\] 
as reproducing kernel Hilbert spaces,  
where we consider $\C$ to be the reproducing kernel Hilbert space consisting of constant functions on $\bH$. 
\end{prop}

\begin{proof}
It follows from the Nevanlinna representation (\ref{eq:5-0}) that 
\begin{equation*}
\frac{\varphi(z)-\overline{\varphi(w)}}{z-\overline{w}}
=c+\frac{1}{\pi}\int_{-\infty}^{\infty}\frac{1}{(t-\overline{w})(t-z)}\ d\mu(t).
\end{equation*}
By this identity and Aronszajn's sums of kernels theorem  
(see Theorem 5.4 in Paulsen--Raghupathi~\cite{PR}), we conclude the proof.
\end{proof}

\begin{thm}\label{lem:6-3}
Let $\varphi$ be a Pick function and $\psi$ be the function defined by (\ref{eq:6-2}). 
Then, 
\[
V: f(z)\mapsto \frac{1-T^{-1}(w)}{1-\psi\circ T^{-1}(w)}f\circ T^{-1}(w)\quad (w=T(z))
\]
gives an isometry from $\cH(\psi)$ onto $\cP(\varphi)$.   
\end{thm}

\begin{proof}
We shall give a proof different from that given in \cite{Yamada}. 
In what follows, we will write $w=T(z)$ and
\[
\frac{1-T^{-1}(w)}{1-\psi\circ T^{-1}(w)}f\circ T^{-1}(w)=
\frac{1-z}{1-\psi(z)}f(z)
\]
for short. 
It follows from (\ref{eq:6-3}) that 
\[
\frac{1-z}{1-\psi(z)}\cdot \dfrac{1-\overline{\psi(\lam)}\psi(z)}{1-\overline{\lam}z}
=\frac{1-\overline{\psi(\lam)}}{1-\overline{\lam}}
\cdot\frac{\varphi(w)-\overline{\varphi(\alpha)}}{w-\overline{\alpha}}
\]
belongs to $\cP(\varphi)$. 
Moreover, by (\ref{eq:6-3}) again, we have that
\begin{align*}
&\left\la V\dfrac{1-\overline{\psi(\lam)}\psi(z)}{1-\overline{\lam}z}, 
V\dfrac{1-\overline{\psi(\mu)}\psi(z)}{1-\overline{\mu}z} \right\ra_{\cP(\varphi)}\\
&=\left\la \frac{1-z}{1-\psi(z)}\cdot \dfrac{1-\overline{\psi(\lam)}\psi(z)}{1-\overline{\lam}z}, 
\frac{1-z}{1-\psi(z)}\cdot\dfrac{1-\overline{\psi(\mu)}\psi(z)}{1-\overline{\mu}z} \right\ra_{\cP(\varphi)}\\
&=\left\la\frac{1-\overline{\psi(\lam)}}{1-\overline{\lam}}
\cdot\frac{\varphi(w)-\overline{\varphi(\alpha)}}{w-\overline{\alpha}},
\frac{1-\overline{\psi(\mu)}}{1-\overline{\mu}}\cdot \frac{\varphi(w)-\overline{\varphi(\beta)}}{w-\overline{\beta}}\right\ra_{\cP(\varphi)}\\
&= \frac{1-\overline{\psi(\lam)}}{1-\overline{\lam}}\cdot \frac{1-\psi(\mu)}{1-\mu}
\cdot\frac{\varphi(\beta)-\overline{\varphi(\alpha)}}{\beta-\overline{\alpha}}\\
&=\frac{1-\overline{\psi(\lam)}\psi(\mu)}{1-\overline{\lam}\mu}\\
&=\left\la \dfrac{1-\overline{\psi(\lam)}\psi(z)}{1-\overline{\lam}z}, 
\dfrac{1-\overline{\psi(\mu)}\psi(z)}{1-\overline{\mu}z} \right\ra_{\cH(\psi)}.
\end{align*}
Thus we have the conclusion.
\end{proof}

\begin{ex}\label{cor:5-5}\rm 
Let $\psi$ be the function defined by a Pick function $\varphi$ and (\ref{eq:6-2}). 
If 
\[
\lim_{r\to 1}\frac{1-r}{1+r}\cdot \frac{1+\psi(r)}{1-\psi(r)}\neq 0,
\] 
then
\[
\frac{1-\psi(z)}{1-z}
\]
belongs to $\cH(\psi)$.
Indeed, since the nonnegative constant $c$ in (\ref{eq:5-0}) is obtained by the well-known formula 
\[
c=\lim_{y\to \infty}\dfrac{\varphi(iy)}{iy},
\]
the assumption is equivalent to that $c\neq 0$. 
Then, it follows from Proposition \ref{thm:5-4} that the constant function $1$ belongs to $\cP(\varphi)$. 
Hence, by Theorem \ref{lem:6-3}, we have the conclusion.
\end{ex}

\section{Chordal L\"{o}wner expansions}\label{sec:2-1-6}

We shall begin with the following fundamental example in the theory of the chordal L\"{o}wner equation. 
\begin{ex}\label{ex:6-1} \rm
We set 
\[
B_{rs}(z)=\sqrt{z^2-2(s-r)}\quad (0\leq r \leq s).
\] 
Then, $B_{rs}$ is a univalent map from $\bH$ onto $\bH \setminus \{iy: 0<y \leq \sqrt{2(s-r)}\}$. 
Moreover,  
$B_{rs}$ is the solution of  
\begin{equation}\label{eq:6-1}
\frac{\p B_{rs}}{\p s}(z)=-\frac{1}{B_{rs}(z)}
\end{equation}
with initial condition $B_{ss}(z)=z$. 
\end{ex}

The differential equation (\ref{eq:6-1}) is the basic case of the chordal L\"{o}wner equation
\begin{equation}\label{eq:6-2-2}
\frac{\p B_{rs}}{\p s}(z)=\int_{\R}\frac{1}{\xi-B_{rs}(z)}\ d\nu_s(\xi)=\int_{\R}\varphi_{\xi}(B_{rs}(z))\ d\nu_s(\xi),
\end{equation}
where $\nu_s$ is a probability measure on $\R$ and $\varphi_{\xi}$ is the function in Example \ref{ex:5-1} 
(see (3.7) in Schlei{\ss}inger~\cite{Sch}).
Indeed, choosing $\nu_s$ as the Dirac measure centered on $\xi=0$ for all $s\geq 0$ in (\ref{eq:6-2-2}), 
we have (\ref{eq:6-1}). 

\begin{lem}\label{lem:6-2}
Let $B_t(z)=B_{at}(z)$ $(0\leq a \leq t \leq b)$ be the solution of 
the chordal L\"{o}wner equation (\ref{eq:6-2-2}) with initial condition $B_{tt}(z)=z$.
Then, for any $\alpha$ in $\bH$, 
\[
\frac{\p}{\p t}
\frac{B_t(z)-\overline{B_t(\alpha)}}{z-\overline{\alpha}}
=\left(\int_{\R}\frac{1}{\left(\xi-\overline{B_t(\alpha)}\right)(\xi-B_t(z))}\ d\nu_t(\xi)\right) 
\frac{B_t(z)-\overline{B_t(\alpha)}}{z-\overline{\alpha}}.
\]
\end{lem}

\begin{proof}
By the chordal L\"{o}wner equation (\ref{eq:6-2-2}), we have the conclusion. 
\end{proof}

Let $\M(t)$ $(0\leq a\leq t \leq b)$ denote the reproducing kernel Hilbert space generated by kernel function 
\[
\int_{\R} \frac{1}{\left(\xi-\overline{B_t(w)}\right)(\xi-B_t(z))}\ d\nu_t(\xi).
\]
Using the notation of Theorem \ref{thm:2-2}, we may write 
$\M(t)=(\bI_{1/(\xi-B_t)};L^2(\R,\nu_t(\xi)))$. 
Moreover, we consider integral operator
\[
\bI: f(t,z) \mapsto \int_a^b f(t,z)\ dt\quad (f\in L^2(\M(t))). 
\]
Then, it follows from Theorem \ref{thm:2-2} that 
\begin{equation}\label{eq:6-3-1}
\int_a^b \left(\int_{\R}\frac{1}{\left(\xi-\overline{B_t(w)}\right)(\xi-B_t(z))}\ d\nu_t(\xi)\right)\ dt
\end{equation}
is the kernel function which generates $\M=\M(\bI;L^2(\M(t)))$.  

\begin{thm}\label{thm:6-3}
Let $B_t(z)=B_{at}(z)$ $(0\leq a \leq t \leq b)$ be the solution of 
the chordal L\"{o}wner equation (\ref{eq:6-2-2}) with initial condition $B_{tt}(z)=z$.
Then, 
\begin{align*}
\cP(B_b)
&=\exp \M(\bI;L^2(\M(t)))\\
&=\bigvee\left\{\exp \int_a^b g(t,z)\ dt : g(t,z)\in L^2(\M(t))\right\}.
\end{align*}
\end{thm}

\begin{proof}
It suffices to show the identity for the corresponding two kernel functions. 
Let $k$ denote the kernel function defined by (\ref{eq:6-3-1}). 
Then, 
for any $\alpha$ in $\bH$, it follows from Lemma \ref{lem:6-2} that   
\begin{align*}
\exp k(z,\alpha)
&=\exp \left(\int_a^b \left(\int_{\R} \frac{1}{\left(\xi-\overline{B_t(\alpha)}\right)(\xi-B_t(z))}\ d\nu_t(\xi)\right)\ dt\right)\\
&=\exp \left(\int_a^b\frac{z-\overline{\alpha}}{B_t(z)-\overline{B_t(\alpha)}}
\cdot \dfrac{\p}{\p t}\dfrac{B_t(z)-\overline{B_t(\alpha)}}{z-\overline{\alpha}}\ dt\right)
\\
&=\exp \left(\int_a^b \dfrac{\p}{\p t}\log \left(\dfrac{B_t(z)-\overline{B_t(\alpha)}}{z-\overline{\alpha}}\right)\ dt\right)\\
&=\frac{B_b(z)-\overline{B_b(\alpha)}}{z-\overline{\alpha}}.
\end{align*} 
\end{proof}

\begin{cor}\label{cor:6-4}\rm 
Let $B_t(z)=B_{at}(z)$ $(0\leq a\leq t \leq b)$ be the family of functions in Example \ref{ex:6-1}. Then, 
\begin{align*}
\cP(B_b)
&=\exp \M(\bI_{1/B_t};L^2([a,b]))\\
&=\bigvee \left\{ \exp\left(\int_a^b\dfrac{1}{B_t(z)}f(t)\ dt\right):f\in L^2([a,b]) \right\}.
\end{align*}
Further, for $\psi_t=T^{-1}\circ B_t \circ T$, 
\[
\cH(\psi_b)=\bigvee \left\{\dfrac{1-\psi_b(z)}{1-z}\exp\left(\int_a^bi\dfrac{\psi_t(z)-1}{\psi_t(z)+1}f(t)\ dt\right): f\in L^2([a,b])  \right\}.
\]
\end{cor}

\begin{proof}
It follows from the assumption that each $\nu_t$ is the Dirac measure centered on $\xi=0$ for all $t\geq 0$. 
Hence, we have that 
\[
\int_a^b \left(\int_{\R}\frac{1}{\left(\xi-\overline{B_t(w)}\right)(\xi-B_t(z))}\ d\nu_t(\xi)\right)\ dt
=\int_a^b \frac{1}{\overline{B_t(w)}B_t(z)}\ dt.
\]
Therefore, we have that 
\[
\cP(B_b)
=\exp \M(\bI_{1/B_t};L^2([a,b]))
\]
by 
Theorem \ref{thm:2-2} and Theorem \ref{thm:6-3}. 
By Theorem \ref{lem:6-3}, we have the representation of $\cH(\psi_t)$. 
\end{proof}

\begin{ex}\rm
Let $B_t(z)=B_{at}(z)$ $(0\leq a\leq t \leq b)$ be the family of functions in Example \ref{ex:6-1}. 
Then, 
\[
\exp(z-B_b(z))
\]
belongs to $\cP(B_b)$. 
Indeed, by the chordal L\"{o}wner equation (\ref{eq:6-1}), we have that 
\begin{align*}
\exp \left(\int_a^b \frac{1}{B_t(z)}\ dt\right)
&=\exp \left(-\int_a^b \frac{\p B_t}{\p t}(z)\ dt\right)\\
&=\exp(z-B_b(z)).
\end{align*}
\end{ex}

\begin{acknowledgment}\rm
The author would like to thank 
Professor Ikkei Hotta (Yamaguchi University) and Dr. Takuya Murayama (Kobe University). 
This research is greatly indebted to their unpublished surveys on the L\"{o}wner theory. 
This work was supported by JSPS KAKENHI Grant Number JP24K06771 and the Research Institute for Mathematical Sciences,
an International Joint Usage/Research Center located in Kyoto University.
\end{acknowledgment}


\begin{thebibliography}{References}

\bibitem{Ando}T. Ando, 
\textit{De Branges spaces and analytic operator functions}. Lecture Notes, Hokkaido University, Sapporo, 1990. 

\bibitem{Conway}
J. B. Conway, 
\textit{Functions of one complex variable}. II.
Grad. Texts in Math., 159. Springer-Verlag, New York, 1995. 

\bibitem{de Branges}
L. de Branges, 
\textit{Hilbert spaces of entire functions}.
Prentice-Hall, Inc., Englewood Cliffs, NJ, 1968. 

\bibitem{de Branges2}
L. de Branges, 
\textit{L\"{o}wner expansions}.
J. Math. Anal. Appl. 100 (1984), no. 1, 323--337.

\bibitem{de Branges3}
L. de Branges, 
\textit{Underlying concepts in the proof of the Bieberbach conjecture}. 
Proceedings of the International Congress of Mathematicians, Vol. 1, 2 (Berkeley, Calif., 1986), 25--42.
American Mathematical Society, Providence, RI, 1987.

\bibitem{dBR}
L. de Branges and J. Rovnyak, 
\textit{Appendix on square summable power series} in  
\textit{Canonical models in quantum scattering theory}. 
Perturbation Theory and its Applications in Quantum Mechanics (Proc. Adv. Sem. Math. Res. Center, U.S. Army, Theoret. Chem. Inst., Univ. of Wisconsin, Madison, Wis., 1965), 295--392. 
John Wiley \& Sons, Inc., New York--London--Sydney, 1966.

\bibitem{dBR2}
L. de Branges and J. Rovnyak, 
\textit{Square summable power series}.
Holt, Rinehart and Winston, New York--Toronto--London, 1966.

\bibitem{Duren}
P. L. Duren, 
\textit{Theory of $H^p$ spaces}.
Pure Appl. Math., 38.
Academic Press, New York--London, 1970. 

\bibitem{FHR}
E. Fricain, A. Hartmann and W. T. Ross, 
\textit{Concrete examples of $\cH(b)$ spaces}.
Comput. Methods Funct. Theory 16 (2016), no. 2, 287--306.

\bibitem{FM}
E. Fricain and J. Mashreghi, 
\textit{The theory of $\cH(b)$ spaces}. Vol. 2.
New Math. Monogr., Cambridge University Press, 21. Cambridge, 2016. 

\bibitem{G1}
S. Ghosechowdhury, 
\textit{An expansion theorem for state space of unitary linear system whose transfer function is a Riemann mapping function}.
Reproducing kernels and their applications (Newark, DE, 1997), 81--95.
Int. Soc. Anal. Appl. Comput., 3.
Kluwer Academic Publishers, Dordrecht, 1999.

\bibitem{G2}
S. Ghosechowdhury, 
\textit{L\"{o}wner expansions}. 
Math. Nachr. 210 (2000), 111--126.

\bibitem{HHM}
T. Hasebe, I. Hotta and T. Murayama, 
\textit{Additive processes on the real line and Loewner chains}, 
arXiv:2412.18742.   

\bibitem{KS}
S. Kuwahara and M. Seto, 
\textit{Exponentials of de Branges--Rovnyak kernels}.
Canad. Math. Bull. 65 (2022), no. 2, 447--455.

\bibitem{LN}
B. {\L}anucha and M. T. Nowak,
\textit{Examples of de Branges--Rovnyak spaces generated by nonextreme functions}.
Ann. Acad. Sci. Fenn. Math. 44 (2019), no. 1, 449--457.

\bibitem{Nik}
N. K. Nikolski,
\textit{Operators, functions, and systems: an easy reading}. Vol. 1. 
Hardy, Hankel, and Toeplitz. Translated from the French by Andreas Hartmann. Mathematical Surveys and Monographs, 92. 
American Mathematical Society, Providence, RI, 2002. 

\bibitem{PR}
V. I. Paulsen and M. Raghupathi, 
\textit{An introduction to the theory of reproducing kernel Hilbert spaces}.
Cambridge Studies in Advanced Mathematics, 152. Cambridge University Press, Cambridge, 2016. 

\bibitem{RR}M. Rosenblum and J. Rovnyak, 
\textit{Topics in Hardy classes and univalent functions}. 
Birkh\"{a}user Verlag, 1994.

\bibitem{Sarason}D. Sarason, 
\textit{Sub-Hardy Hilbert spaces in the unit disk}. 
University of Arkansas Lecture Notes in the Mathematical Sciences, 10. 
A Wiley-Interscience Publication. John Wiley \& Sons, Inc., New York, 1994. 

\bibitem{Sch}
S. Schlei{\ss}inger, 
\textit{The chordal Loewner equation and monotone probability theory}. 
Infin. Dimens. Anal. Quantum Probab. Relat. Top. 20 (2017), no. 3, 1750016.  

\bibitem{Seto}
M. Seto, 
\textit{A Fock space approach to the theory of kernel functions}.
Positivity 28 (2024), no. 1, Paper No. 9. 

\bibitem{S}
L.\ Shulman, 
\textit{Perturbations of unitary transformations}.
Amer. J. Math.   91 (1969), 267--288.

\bibitem{VN1}V. I. Vasyunin and N. K. Nikol'ski\u{\i}, 
\textit{Quasiorthogonal decompositions with respect to complementary metrics, and estimates of univalent functions}. 
Leningrad Math.\ J. 2 (1991), no. 4, 691--764. 

\bibitem{VN2}
V. I. Vasyunin and N. K. Nikol'ski\u{\i}, 
\textit{Operator-valued measures and coefficients of univalent functions}.
St. Petersburg Math. J. 3 (1992), no. 6, 1199--1270.

\bibitem{Yamada}A. Yamada,   
\textit{Simultaneous operator equations, Schwartz reproducing kernel spaces and de Branges spaces}. 
General topics on applications of reproducing kernels,  
S\={u}rikaisekikenky\={u}sho K\={o}ky\={u}roku, 
no. 1980 (2016), 166--183.
\end{thebibliography}
\end{document}